\documentclass[a4paper]{amsart}
\usepackage{amsthm,amsfonts, latexsym, eucal, amsmath, amssymb, mathrsfs, hyperref, graphicx, xcolor, verbatim, overpic}
\usepackage[latin1]{inputenc}
\pdfoutput=1

\usepackage[margin=1.1in]{geometry}

 \newtheorem{thm}{Theorem}
 \newtheorem{prop}[thm]{Proposition}
 \newtheorem{lemma}[thm]{Lemma}
 \newtheorem{cor}[thm]{Corollary}

 \theoremstyle{definition}
 \newtheorem{definition}[thm]{Definition}
 \newtheorem{ex}[thm]{Example}

 \theoremstyle{remark}
 \newtheorem{remark}[thm]{Remark}
\numberwithin{thm}{section}
\numberwithin{equation}{section}

\def\Spec{{\rm Spec}\,}
\def\Spf{{\rm Spf}}

\def\rk{{\rm rk}}
\def\defect{{\rm def}}

\def\dom{{\rm dom}}
\def\codim{{\rm codim}}
\def\Hom{{\rm Hom}}
\def\dim{{\rm dim}\,}

\def\defect{{\rm def}}

\def\height{\operatorname{ht}}

\def\N{{N}}
\def\R{{\mathbb R}}

\def\F{{\mathbb F}}
\def\Q{{\mathbb Q}}

\def\O{{\mathcal O}}

\def\cd{cordial}
\def\Cd{Cordial }
\def\sha{small-height-avoiding}

\definecolor{ltgreen}{rgb}{0.0, 0.5, 0.0}
\definecolor{dkgreen}{rgb}{0.0, 0.42, 0.24}

\begin{document}

\begin{title}
{Generic Newton points and the Newton poset in Iwahori-double cosets}
\end{title}
\author{Elizabeth Mili\'{c}evi\'{c}}
\author{Eva Viehmann}

\address{Elizabeth Mili\'{c}evi\'{c}, Department of Mathematics \& Statistics, Haverford College, 370 Lancaster Avenue, Haverford, PA, 19041, USA}
\email{emilicevic@haverford.edu}

\address{Eva Viehmann, Technische Universit\"at M\"unchen\\Fakult\"at f\"ur Mathematik - M11 \\ Boltzmannstr. 3\\85748 Garching bei M\"unchen\\Germany}
\email{viehmann@ma.tum.de}

\date{}

\begin{abstract}{We consider the Newton stratification on Iwahori-double cosets in the loop group of a reductive group. We describe a group-theoretic condition on the generic Newton point, called cordiality, under which the Newton poset (i.e.~the index set for non-empty Newton strata) is saturated and Grothendieck's conjecture on closures of the Newton strata holds. Finally, we give several large classes of Iwahori-double cosets for which this condition is satisfied by studying certain paths in the associated quantum Bruhat graph.}
\end{abstract}

\maketitle

\section{Introduction}\label{intro}

Let $G$ be a reductive group over a finite field $\mathbb F_q$. In their seminal paper \cite{DL}, Deligne and Lusztig described the representations of the finite group $G(\mathbb F_q)$ by realizing them in the cohomology of suitable coverings of a family of varieties indexed by the elements of the finite Weyl group $W$ of $G$. These so-called Deligne-Lusztig varieties are locally closed subschemes of the flag variety associated with $G$.  In addition, they are smooth and have dimension equal to the length $\ell(w)$ of the element $w \in W$.

We consider affine Deligne-Lusztig varieties, which are an analog of the above in affine flag varieties. Now let $G$ be a quasi-split reductive group defined over $F=\F_q(\!(t)\!)$. In contrast to the classical case considered by Deligne and Lusztig, affine Deligne-Lusztig varieties depend on two parameters: an element of the Iwahori-Weyl group $\widetilde W$ of $G$, and an element $b\in G(\breve F)$ where $\breve F=\overline{\F}_q(\!(t)\!)$. To define these varieties, we consider the affine Bruhat decomposition.  For $I$ an Iwahori subgroup of $G$ and $\breve \O = \overline{\F}_q[\![t]\!]$ the ring of integers in $\breve F$, we have  $$G(\breve F)=\coprod_{x\in \widetilde W}I(\breve \O)xI(\breve \O).$$ Let $\sigma$ denote the Frobenius automorphism of $\breve F$ over $F$ mapping all coefficients of the Laurent series to their $q^{\text{th}}$ powers.  The affine Deligne-Lusztig variety associated to $x\in \widetilde W$ and $b\in G(\breve F)$ is defined as the locally closed reduced subscheme $X_x(b)$ of the affine flag variety of $G$ given by $$X_x(b)(\overline {\F}_q)=\{g\in G(\breve F)/I(\breve \O)\mid g^{-1}b\sigma(g)\in IxI\}.$$ 

Besides being a natural analog of the classical theory, these varieties play an important role when studying the special fiber of both Shimura varieties and moduli spaces of shtukas. More precisely, affine Deligne-Lusztig varieties describe the Kottwitz-Rapoport stratification of the special fiber of Rapoport-Zink moduli spaces \cite{RapShimura}. However, their geometry is much harder to understand than in the classical case. Even the question of whether $X_x(b)\neq \emptyset$ for a given pair $(x,b)$ is a notoriously difficult problem and remains unsettled. Questions on the geometry of $X_x(b)$ are closely related to understanding the intersection of $IxI$ with the $\sigma$-conjugacy class $$[b]=\{g^{-1}b\sigma(g)\mid g\in G(\breve F)\}.$$
These $\sigma$-conjugacy classes $B(G)=\{[b]\mid b\in G(\breve F)\}$ were classified by Kottwitz in \cite{KotIsoII} by two invariants: the Newton point and the Kottwitz point. The set $B(G)$ has a partial order, induced by requiring equality of the Kottwitz points, and using the dominance order on Newton points. 

Denote by $B(G)_x$ the subset of those $\sigma$-conjugacy classes that meet the given double coset $IxI$. A necessary condition for $[b]\in B(G)_x$ is that the Kottwitz points of $b$ and $x$ coincide.  However, this necessary condition is far from sufficient, and a complete description of $B(G)_x$ is only known in very special cases. Whenever it is non-empty, $\mathcal N_{[b],x}:= [b]\cap IxI$ is shown in \cite{RR} to be the set of geometric points of a locally closed reduced subscheme of $IxI$; namely, the Newton stratum associated with $[b]$.  Another natural (and similarly unsolved) question is to describe the closure of $\mathcal N_{[b],x}$ in $IxI$.

This situation is in stark contrast to the case in which $I$ is replaced by a hyperspecial maximal compact subgroup $K$ of $G$ (assuming that this exists). Here, the $K$-double cosets in $G(\breve F)$ are indexed by the dominant cocharacters $\mu$ of a fixed maximal torus of $G$. For a given $\mu$, the set of $\sigma$-conjugacy classes meeting $K\mu(t)K$ is non-empty if and only if the Newton point $\nu(b)$ is less than or equal to $\mu$ in dominance order, in addition to the obvious criterion that the Kottwitz points of $b$ and $\mu$ coincide; see \cite{KRFcrystals, Luc, Gashi3}. Furthermore, the similarly defined affine Deligne-Lusztig varieties are equidimensional of known dimension, and the closure of a Newton stratum is equal to the union of all Newton strata associated with $[b']\leq [b]$; compare \cite{grothconj}. We remark that, strictly speaking, some of these results are only shown under the additional assumption that $G$ is split. However, the same proofs also work without this additional assumption.

In the case of Iwahori-double cosets, however, none of these properties hold in general.  In particular, the Newton strata  are not equidimensional \cite[Sec.~5]{GHdim}, it is unknown under which conditions the closure of any Newton stratum is a union of strata, and we do not even have a general conjecture describing the set $B(G)_x$. The goal of this paper is to identify a large family of elements $x \in \widetilde{W}$ for which both the non-emptiness pattern for the intersections $[b]\cap IxI$ and the closure relations of Newton strata resemble the known picture for hyperspecial subgroups $K$.

\subsection{Statements of the main theorems}\label{sec:mainthms}

One element of the Newton poset $B(G)_x$ that is of particular interest is the unique maximal element, which coincides with the generic $\sigma$-conjugacy class $[b_x]$ in the irreducible double coset $IxI$. There are descriptions of $[b_x]$ that give finite algorithms to compute it, but they are not themselves closed formulas; see \cite[Cor.~5.6]{VieTrunc} and \cite[Thm.~3.2]{BeMaxNPs}.

We define an element $x \in \widetilde{W}$ to be \emph{\cd}~if it satisfies the equality
\begin{equation*}
\ell(x) - \ell(\eta(x)) = \langle 2\rho, \nu_x \rangle - \defect (b_x),
\end{equation*}
where $\rho$ is the half-sum of the positive roots and $\nu_x = \nu(b_x)$ is the generic Newton point in $IxI$. See Definition \ref{defcd} for a formal discussion of cordial elements, and refer to Section \ref{sec:notation} for the definitions of the map $\eta$ and the defect. To introduce our first main theorem, we comment on the chosen terminology. The notion \cd~refers to the fact (explained in Section \ref{sec2}) that satisfying the above equality is equivalent to the condition that the dimension of the affine Deligne-Lusztig variety $X_x(b_x)$ agrees with its virtual dimension in the sense of \cite{HeAnnals}. The cordial condition is thus equivalent to the condition that this variety ``has the \emph{cor}rect \emph{di}mension''.  Moreover, the following theorem illustrates that the cordial condition also gives rise to especially ``well-behaved'' geometry for the associated Newton strata.

\begin{thm}\label{thmmaincd}
Let $x$ be \cd. Then $B(G)_x$ is saturated, and for $[b]\in B(G)_x$ we have
\begin{enumerate}
\item[(a)] $\mathcal N_{[b],x}$ is equidimensional, and its codimension in $IxI$ is equal to the maximal length of any chain from $[b]$ to $[b_x]$ in $B(G)_x$ \emph{(}or, equivalently, in $B(G)$\emph{)}.
\item[(b)] $\overline {\mathcal N_{[b],x}}$ is the union of all $\mathcal N_{[b'],x}$ with $[b']\in B(G)_x$ and $[b']\leq [b]$.
\end{enumerate}
\end{thm}
\noindent Here a subset $S$ of $B(G)$ is called saturated if for any $[b_1]\leq [b_2]\leq [b_3]$ in $B(G)$ such that $[b_1],[b_3]\in S$, we also have $[b_2]\in S$.

Theorem \ref{thmmaincd} gives a condition that can be checked from the maximal element of $B(G)_x$ alone, but implies that the shape of the entire poset $B(G)_x$, as well as all dimensions and closures of the Newton strata within $IxI$, behave as nicely as the Newton strata for $K$-double cosets. The only difference that may occur is that the set $B(G)_x$ does not in general contain all elements of the form $\{[b]\in B(G) \mid [b] \leq [b_x]\}$; small elements up to a certain lower bound (discussed in \cite{VieMinNP}) may be missing. In fact, one can also show a stronger statement, which has assumptions that are more difficult to check in general; compare Theorem \ref{thm1strong}. In Theorem \ref{thmconv}, we also prove a partial converse of Theorem \ref{thmmaincd}, showing that non-cordial elements cannot share all of these same good geometric properties.

Our next theorem explicitly identifies several families of \cd ~elements. For sufficiently low-rank groups, it is sometimes possible to directly calculate the Newton poset $B(G)_x$ for every $x \in \widetilde{W}$.  For example, all of the questions we address in this paper can be settled for the group $G={\rm SL}_3$ using the first author's thesis \cite{BeThesis}. For this group, an element $x$ is cordial if and only if $B(G)_x$ is saturated, and one can give a complete description of the set of cordial elements. Further, for $G={\rm SL}_3$ all Newton strata are equidimensional, and part (b) of Theorem \ref{thmmaincd} also holds in all cases.  See Example \ref{SL3Ex} for more details.

In general, it appears to be a fairly difficult problem to fully characterize the \cd~elements in a manner which does not require specific knowledge of the generic Newton point, but we provide several  interesting families of \cd~elements in the following theorem. 

\begin{thm}\label{T:MainEx} 
Assume that $G$ is quasi-split. Let $x=t^{v\lambda}w \in \widetilde{W}$ with $\lambda$ dominant.
\begin{enumerate}
\item[(a)] If $x$ is in the antidominant Weyl chamber in which case $v=w_0$, then $x$ is cordial.
\end{enumerate}
Now suppose that $G$ is split, connected, and semisimple. 
Assume that for all simple roots $\alpha_i$ we have
$\langle \alpha_i, \lambda \rangle > M,$ where $M$ is a fixed constant depending on $G$ and $x$.
\begin{enumerate}
\item[(b)] If any reduced expression for $\eta(x)= v^{-1}wv \in W$ 
uses each simple reflection at most once, then $x$ is \cd.
\item[(c)] If $x$ is in the dominant Weyl chamber in which case $v=1$, then $x$ is cordial if and only if every reduced expression for $\eta(x)=w$
 avoids all non-simple reflections $s_\alpha$ such that $\ell(s_\alpha) = \langle 2\rho, \alpha^\vee \rangle -1$. 
\end{enumerate}
\end{thm}

The hypotheses on $G$ in parts (b) and (c) of Theorem \ref{T:MainEx}  are a direct reflection of the reliance upon the first author's formula for calculating the generic Newton point via the quantum Bruhat graph \cite{BeMaxNPs}, which is stated in precisely this level of generality.  The additional hypothesis on the coroot $\lambda$ which keeps $x$ sufficiently far from the walls of any Weyl chamber is referred to as \emph{superregularity}; see \cite{BeMaxNPs} for precise formulas for the constant $M$ and related discussion. Under this superregularity hypothesis, we characterize cordiality purely in terms of calculating lengths of certain paths in the quantum Bruhat graph; see Proposition \ref{etaminpath}.
Those elements which use each simple reflection at most once as in (b) are called \emph{standard parabolic Coxeter elements} in Definition \ref{parcoxeter}. We make the condition appearing in (c) precise in Definition \ref{sha}, where we refer to those elements as \emph{\sha}; see Section \ref{spcsha} for further discussion of this terminology and related properties. 

\begin{remark}
Throughout the paper we assume that the local field $F$ is of equal characteristic $p$. Instead, one can also study the analogous questions replacing $F$ by the fraction field of the Witt ring $W(\mathbb F_q)$. In that case, the loop group and the affine flag variety are perfect schemes; see \cite{Zhu, BhattScholze}. Our results on cordiality of certain elements (in particular Theorem \ref{T:MainEx}), as well as on the non-emptiness and dimensions of affine Deligne-Lusztig varieties, can be directly translated to corresponding results in the arithmetic case, and thus they still hold in that context. Indeed, all of these assertions can be translated into properties of associated elements of the extended affine Weyl group, and degrees of certain class polynomials via \cite[Sec.~6]{HeAnnals}. These results are then independent of the characteristic of the chosen field. On the other hand, the proof of Theorem \ref{thmmaincd} that we give here does not directly generalize to mixed characteristic. In particular, there are currently no analogues for our assertions concerning closure relations of the Newton strata.

Furthermore, we assume that $G$ is quasi-split. Using the argument in \cite[Sec.~2]{GoertzHeNie}, one can reduce the question of determining non-emptiness and dimensions of affine Deligne-Lusztig varieties for a general connected reductive group over $F$ to those for the quasi-split inner form of its adjoint group. In this way, our results imply corresponding results for these more general groups.
\end{remark}

\vskip 5pt

\noindent{\it Acknowledgments.} The first author was partially supported by National Science Foundation grant 1600982 ``RUI: Affine flags, $p$-adic representations, and quantum cohomology''. The second author was partially supported by European Research Council Starting Grant 277889 ``Moduli spaces of local $G$-shtukas'', and later by European Research Council Consolidator Grant 770936 ``Newton strata''. The collaborative visit during which this project began was supported by both this ERC Starting Grant held by the second author, as well as by Simons Collaboration Grant 318716 held by the first author.  Part of this work was carried out while the first author was a Visiting Scientist at the Max-Planck-Institut f\"{u}r Mathematik, and a Director's Mathematician in Residence at the Budapest Semesters in Mathematics, jointly supported by the R\'enyi Alfr\'ed Matematikai Kutat\'oint\'ezet. We thank Ulrich G\"ortz for a helpful discussion, Xuhua He for his comments on a previous version of this work, and Petra Schwer and Anne Thomas for early discussions of an alternative method using the Deligne-Lusztig galleries from \cite{MST}. The first author gratefully acknowledges many useful conversations about special cases with Elizabeth Rule and Noah Weinstein. Finally, the authors thank the anonymous referee for their helpful suggestions.

\section{Notation}\label{sec:notation}

Let $F$ be a local field of characteristic $p$ with ring of integers $\mathcal O_F$, uniformizer $t$ and residue field $\mathbb{F}_q$; that is, $F\cong \F_q(\!(t)\!)$. Let $\breve F\cong \overline{\F}_q(\!(t)\!)$ denote the completion of the maximal unramified extension of $F$, and $\breve{\mathcal O}$ its ring of integers. Let $\sigma$ denote the Frobenius of $\breve F$ over $F$.  

Let $G$ be a quasi-split connected reductive group over $F$. Let $S$ be a maximal $\breve{F}$-split torus of $G$ which contains a maximal split torus. Let $T=C_G(S)$ be its centralizer, a maximal torus. Let $\mathcal A$ be the apartment of the Bruhat-Tits building of $G_{\breve F}$ corresponding to $S_{\breve F}$. Then the $\sigma$-action on the building preserves $\mathcal A$. Let $\mathfrak a$ be an alcove in $\mathcal A$ that is fixed by $\sigma$. Let $I$ be the corresponding Iwahori subgroup.

Let $N_T$ denote the normalizer of $T$. Then the (relative) Weyl group $W$ of $G$ is defined as $N_T(\breve F)/T(\breve F)$. The Iwahori-Weyl group is $\widetilde W=N_T(\breve F)/(T(\breve F)\cap \breve I)$.  The Frobenius $\sigma$ of $\breve F$ over $F$ acts on $G(\breve F)$ and also induces an automorphism of $\widetilde W$, which we denote again by $\sigma$. We fix a special vertex of the base alcove $\mathfrak a$. This induces a splitting of the natural projection $\widetilde W\rightarrow W$, and an isomorphism $\widetilde W\cong X_*(T)_{I_F}\rtimes W$ where $I_F$ is the inertia subgroup of the absolute Galois group of $F$. We then have $$G(\breve F)=\coprod_{x\in \widetilde W}I(\breve \O)xI(\breve \O).$$ Here, for every $x\in\widetilde W$ we choose a representative in $G(\breve F)$ which we denote again by $x$.

To define a length function $\ell$ on $\widetilde W$ let $W_a$ be the affine Weyl group. We have $\widetilde W\cong W_a\rtimes \Omega$ where $\Omega$ is the normalizer of the base alcove. Since we fixed $I$, we obtain a length function $\ell$ and a Bruhat order $\leq$ on the infinite Coxeter group $W_a$. We extend $\ell$ and $\leq$ from $W_a$ to $\widetilde W$ by setting $\ell(\omega)=0$ for $\omega\in \Omega$, and defining $x\leq y$ if and only if $x$ and $y$ are of the form $x'\omega$ and $y'\omega$ for some $\omega\in \Omega$ and $x'\leq y'\in W_a$.

We choose the dominant Weyl chamber to be the Weyl chamber containing $\mathfrak a$, and let $B$ be the corresponding Borel subgroup of $G$.  Write $x\in \widetilde W$ as $x=v\lambda(t)v^{-1}w=:t^{v\lambda}w$ where $v,w\in W$, $\lambda\in X_*(T)$, and $t^{\lambda}v^{-1}w$ maps the base alcove to the dominant chamber. The map $\eta:\widetilde W\rightarrow W$ is then defined by $\eta(x)=\sigma^{-1}(v^{-1}w)v$. Notice that in the literature, this map is sometimes also denoted $\eta_{\delta}$ to emphasize that $\sigma$ induces an automorphism $\delta$ of $\widetilde W$, which is used in the definition of $\eta$.

For $b\in G(\breve F)$, let $[b]=\{g^{-1}b\sigma(g)\mid g\in G(\breve F)\}$ be its $\sigma$-conjugacy class, and let $B(G)$ denote the set of $\sigma$-conjugacy classes. The elements $[b]\in B(G)$ are classified by two invariants (compare \cite{KotIsoI, KotIsoII, RR}).  The first is the Newton point $\nu(b)\in N(G)=(\Hom(\mathbb D_{\overline F},G_{\overline F})/G(\overline F))^{\Gamma}$, where $\mathbb D$ is the pro-torus with character group $\Q$, where $G(\overline F)$ acts by conjugation, and where $\Gamma$ denotes the absolute Galois group of $F$. We have an identification with the set of dominant rational Galois-invariant cocharacters, $N(G)=(X_*(T)_{\Q}/W)^{\Gamma}=X_*(T)_{\Q,\dom}^{\Gamma}$.
The second invariant is the Kottwitz point $\kappa_G(b)\in \pi_1(G)_{\Gamma}$. Here, $\pi_1(G)$ is the quotient of $X_*(T)$ by the coroot lattice. The two invariants $\nu(b)$ and $\kappa_G(b)$ have the same image in $\pi_1(G)_{\Gamma,\Q}$.  For a fixed $x\in\widetilde W$, denote the set $B(G)_x$ of $\sigma$-conjugacy classes $[b]\in B(G)$ such that $\mathcal N_{[b],x}:= [b]\cap IxI\neq \emptyset$. 

The defect of $[b]\in B(G)$ is defined as $\defect(b) = \rk_F G-\rk_F J_b$ where $J_b$ is the reductive group over $F$ with $$J_b(F)=\{g\in G(\breve F)\mid gb=b\sigma(g)\}.$$ 
 There is a partial ordering $\leq$ on $B(G)$ defined by $[b]\leq [b']$ if $\kappa_G(b)=\kappa_G(b')$ and $\nu(b)\leq \nu(b')$; i.e.\ the difference $\nu(b')-\nu(b)$ is a non-negative linear combination of positive coroots; compare \cite[Sec.~2]{RR}.

\section{Maximal Newton points and \cd~elements} \label{sec2}

The aim of this section is to formally define \cd ~elements and to prove Theorem \ref{thmmaincd}. Fix $x\in\widetilde W$ and $b\in G(\breve F)$. The associated affine Deligne-Lusztig variety is defined to be the locally closed, reduced subvariety  $X_x(b)$ of the affine flag variety with $$X_x(b)(\overline{\mathbb F}_q)=\{g\in G(\breve F)/I(\breve \O)\mid g^{-1}b\sigma(g)\in I(\breve \O)xI(\breve \O)\}.$$  In the first subsection, we compute the dimension of an affine Deligne-Lusztig variety $X_x(b)$ in terms of the dimension of the corresponding Newton stratum in $IxI$. In the second subsection, we compare the expression obtained in this way to He's virtual dimension of the same affine Deligne-Lusztig variety. If these two dimensions agree for the generic $\sigma$-conjugacy class in $IxI$, we will call the element $x$ \cd .

\begin{definition}
Given an element $x \in \widetilde{W}$, let $[b_x]$ be the $\sigma$-conjugacy class in the (unique) generic point of $IxI$, and thus the unique maximal element in $B(G)_x$ with respect to the partial ordering on $B(G)$. 

We define the {\emph{Newton point}} $\nu(x)$ of $x$ to be the Newton point $\nu(\dot x)$ of $[\dot x]$ for any representative $\dot x$ of $x$ in $G(F)$. The {\emph{maximal Newton point}} $\nu_x$ of $x$ is then defined as the Newton point of $[b_x]$.
\end{definition}

By definition, $\nu_x$ satisfies $\lambda \leq \nu_x$ for all Newton points $\lambda$ of elements of $B(G)_x$.
The first concrete description of the maximal element of $B(G)_x$ was given by the second author \cite[Cor.~5.6]{VieTrunc}, a weaker version of which can be expressed as $$\nu_x = \max \{ \nu(y) \mid y \in \widetilde{W},\  y \leq x \}.$$ Here the maximum is taken with respect to dominance order, and the elements $y$ and $x$ are related by Bruhat order. Note that this yields a finite algorithm to compute $\nu_x$, but not a closed formula. A slightly more explicit description of $\nu_x$ provided by the first author \cite[Thm.~3.2]{BeMaxNPs} is discussed in Section \ref{sec3}, albeit under an additional superregularity hypothesis on $\lambda$, and for split $G$.

\subsection{Comparing dimensions of Newton strata and affine Deligne-Lusztig varieties}

Although we do not dispose of a closed formula for $[b_x]$ here, we can relate its Newton point $\nu_x$ to the dimension of the corresponding affine Deligne-Lusztig variety. 

\begin{lemma}\label{lemdimgen}
Let $x\in \widetilde W$. Then $X_x(b_x)$ is equidimensional with $$\dim~X_x(b_x)=\ell(x)-\langle 2\rho, \nu_x \rangle .$$
\end{lemma} 
For the proof of Lemma \ref{lemdimgen}, we develop a more general theory for all Newton strata that will also be helpful later on. Notice that the dimension formula itself is known by \cite[Theorem 2.23]{He-CDM} using a completely different proof. However, we will crucially use the equidimensionality for our applications. The rough idea of our proof is to express dimensions of affine Deligne-Lusztig varieties using a product structure up to finite morphism on a corresponding Newton stratum. The construction closely follows the corresponding theory for hyperspecial maximal subgroups of \cite{viehmann-wu}. We therefore replace most proofs by references to the corresponding arguments given in loc.~cit. Let us first recall some well-known notions for subschemes of loop groups. Let $LG$ be the loop group associated with $G$. It is defined as the ind-group
scheme representing the functor $R\mapsto G(R(\!( t)\!))$ on the category of $\mathbb F_q$-algebras.

\begin{definition} Let 
$\mathcal{B} $ be a subscheme of the loop group $LG$.
\begin{enumerate}
\item Let $x\in \widetilde W$. Then $\mathcal B$ is \emph{bounded by $x$} if it is contained in the closure of  $IxI$ in $LG$. It is \emph{bounded} if it is contained in a finite union of double cosets $IxI$.
\item Let $I_n$ be the kernel of the projection map $I\rightarrow I(\O_F/(t^n))$. Then $\mathcal B$ is \emph{admissible} if there is an $n\in \mathbb{N}$ with $\mathcal{B}I_n=\mathcal{B}$.
\item For a bounded and admissible algebraic set with $XI_n=X$ let $$\dim X:=\dim(X/I_n) -n\cdot\dim(G).$$
\end{enumerate}
\end{definition}
Notice that this notion of dimension is normalized in a different way than the one in \cite{viehmann-wu}.
\begin{remark}\label{rem221} We can make several initial observations about subschemes of $LG$.
\begin{enumerate}
\item Let $\mathcal{B}$ be bounded. Then one easily sees that $\mathcal{B}$ is admissible if and only if there is an $n'\in \mathbb{N}$ with $I_{n'}\mathcal{B}=\mathcal{B}$. Here $n'$ can be given in terms of the bound for $\mathcal B$ and the integer $n$ arising in the definition of admissibility.
\item The dimension of a bounded and admissible subscheme of $LG$ is independent of the choice of $n.$ 
\item Similarly, one can define the codimension of a closed irreducible admissible subscheme $\mathcal B'$ of some bounded and admissible scheme $\mathcal B$. If $\mathcal B$ is also equidimensional, one easily sees that this codimension agrees with $\dim\mathcal B-\dim \mathcal B'$.
\end{enumerate}
\end{remark}

\begin{prop}\label{propvw23}
 Let $\mathcal{B}$ be a bounded subset of $LG(k)$. Then there is an integer $c\in \mathbb{N}$ such that for each $d\in \mathbb{N}$, each $g\in \mathcal{B}$ and $h\in I_{d+c}(k)$, there is an $l\in I_d(k)$ with $gh=l^{-1}g\sigma(l).$
\end{prop}
\begin{proof} 
Any bounded subset is contained in a finite union of Iwahori-double cosets $IxI$. Each $B(G)_x$ being finite, this implies that $\mathcal B$ only meets finitely many $[b]\in B(G)$. Considering the intersection with each $\sigma$-conjugacy class and each Iwahori double coset separately, we may assume that $\mathcal B\subset [b]\cap IxI$.

By \cite[Thm.~3.7]{HeAnnals}, there is a straight element $y_b\in \widetilde W$ whose representative in $G$ lies in $[b]$. Notice that the additional assumptions on $G$ made in loc.~cit.~are not necessary for the proof of this statement. By \cite[Thm.~1.3]{Nie}, $y_b$ is also $P$-fundamental for some semi-standard parabolic subgroup $P$ of $G$.

By \cite[Prop.~5.1]{HamVie_Finiteness}, there is a bounded subset $\mathcal C$ of $LG$ such that every element of $\mathcal B$ is $\sigma$-conjugate by an element of $\mathcal C$ to $y_b$. In this situation, the same argument as in the last paragraph of the proof of \cite[Thm.~10.1]{HartlVie1} reduces the proof to showing the following claim: For every $n\in\N$ and every $g\in I_n$ there is an $h\in I_n$ with $y_bg=h^{-1}y_b\sigma(h)$. This in its turn is shown in the same way as the statement for $n=0$, see the proof of \cite[Thm.~6.3.1]{GHKRadlvs}.
\end{proof}

\begin{cor}\label{cor26}
Let $b\in IxI$. Then the $I$-$\sigma$-conjugacy class $\mathcal C_b=\{ib\sigma(i)^{-1}\mid i\in I\}$ of $b$ is contained in $IxI$, admissible, and a smooth and locally closed subscheme of $LG$. Further, $\mathcal N_{[b],x}$ is admissible.
\end{cor}
\begin{proof}
In both cases, admissibility follows from the previous proposition. Since $\mathcal C_b$ is one $I$-orbit, it is smooth and locally closed.
\end{proof}
The last assertion in Corollary \ref{cor26} also follows from a corresponding assertion on Newton strata in the whole loop group by He; see \cite[Thm.~A.1]{HeKR}.

\begin{definition}
Let $n\in\mathbb N$ and let $b\in \overline{IxI}$. Here $\overline{IxI}=\bigcup_{x'\leq x}IxI$ denotes the closure of $IxI$ in $LG$. We consider the following functor on the category $(Art/k)$ of Artinian local $k$-algebras with residue field $k$.
\begin{align*}
\mathscr{D}ef(b)_n:(Art/k)\to{}&(Sets),\\
  A\mapsto{}&\{\tilde b\in (\overline{IxI})(A) \text{ with }\tilde b_k=b\}/_{\cong_n}.
\end{align*}
Here, $\tilde b\cong_n \tilde b'$ if there exists a $g\in I_n(A)$ with $g_k=1$ such that $\tilde b'=g^{-1}\tilde b\sigma(g).$
We call $\mathscr{D}ef(b)_n$ the deformation functor of level $n$ of $b$.
\end{definition}

\begin{prop}\label{propdef}
The functor $\mathscr{D}ef(b)_n$ is pro-represented by the formal completion of $I_n\backslash \overline{IxI}$ at $I_nb$, which we denote by $D_{b,n}$. 
\end{prop}
\begin{proof} This is shown completely analogously to the proof of \cite[Prop.~2.9]{viehmann-wu}. 
\end{proof}

For a given $n$, consider the projection morphism $D_{b,n}\rightarrow D_{b,0}$. The projection $LG\rightarrow I\backslash LG$ has sections \'etale locally. Thus we also have a (non-unique) section $s:D_{b,0}\rightarrow D_{b,n}$. Let $(D_{b,0}\times (I_n\backslash I))^\wedge$ denote the completion of $D_{b,0}\times (I_n\backslash I)$ at $(b,1)$. Using that $I_n$ is normal in $I$, the section $s$ induces a (still non-canonical) morphism $$\phi:(D_{b,0}\times (I_n\backslash I))^\wedge\to D_{b,n}$$ given
by $(b,g)\mapsto g^{-1}s(b)\sigma(g)$. 

\begin{lemma}\label{lem27}
 The morphism $\phi:(D_{b,0}\times (I_n\backslash I))^\wedge\to D_{b,n}$ is an isomorphism.
\end{lemma}
\begin{proof} Compare the proof of \cite[Lem.~2.10]{viehmann-wu}.
\end{proof}

Recall from \cite[Lem.~2.11]{viehmann-wu} that for any admissible $\mathbb{F}_q$-algebra $R$ with filtered index poset $\mathbb{N}_0$, the pullback by the natural morphism $\Spf R\rightarrow \Spec R$ induces a bijection between the $\Spec R$-valued points and the $\Spf R$-valued points of $\overline{IxI}$.
Thus we can associate with the formal scheme $D_{b,n}$ a scheme $D_{b,n}'$, and we have a section $D_{b,n}'\rightarrow LG$. In particular, we can study the Newton stratification on $D_{b,n}'.$ For large $n$, Corollary \ref{cor26} implies that the Newton stratification does not depend on the choice of the lift.

For $[b]\in B(G)$, let $y_b\in \widetilde W$ be as in the proof of Proposition \ref{propvw23} a $P$-fundamental element with $[y_b]=[b]$, for some semistandard parabolic subgroup $P$. Its Levi subgroup $M$ containing $T$ centralizes the $M$-dominant Newton point $\nu$ of $y_b.$ From the definition of $P$-fundamental alcoves one can then easily see that $y_b$ is also $\tilde P$-fundamental for $\tilde P=\tilde M P\supset P$ where $\tilde M$ is the centralizer of $\nu$, so $\tilde P$ and $y_b$ are as in the theorem below. If $y$ is $P$-fundamental for some parabolic $P$, let $\overline N$ be the unipotent radical of the opposite parabolic, and let $I_{\overline N}=I\cap L\overline N$. Then by definition of $P$-fundamental alcoves, we have $y^{-1}I_{\sigma^{-1}(\bar{N})}y\subseteq I_{\bar{N}}$.

\begin{thm}\label{thmfoliat}
 Let $b\in IxI$. Let $\mathcal{N}_{[b]}$ be the Newton stratum of $[b]$ in $\Spec D_{b,0}'.$ Let $y_b$ be a $P$-fundamental alcove associated with $[b]$, where $P$ is chosen such that the Levi subgroup $M$ of $P$ containing $T$ equals the centralizer of the $M$-dominant Newton point of $y_b.$ Then there is a finite surjective morphism $$((X_{x}(b)\times_k I_{\bar{N}}/ y_b^{-1}I_{\sigma^{-1}(\bar{N})}y_b)^\wedge)' \to \mathcal{N}_{[b]}.$$ Here again, $(\cdot)'$ denotes the scheme associated with the formal scheme obtained by completion.
  Furthermore, the locus in $\Spec D_{b,0}'$ of elements $I$-$\sigma$-conjugate to $b$ is smooth and equal to the image of $(\{1\}\widehat{\times}_k I_{\bar{N}}/ y_b^{-1}I_{\sigma^{-1}(\bar{N})}y_b)^\wedge)'$ in $\mathcal{N}_{[b]}.$
\end{thm}

\begin{proof}
This follows from essentially the same proof as \cite[Thm.~2.9]{viehmann-wu}, which in  turn was a natural generalization of the proof of Theorem 6.5 or Theorem 6.6 from \cite{HartlVie2} to unramified groups.
\end{proof}

Recall that $\mathcal{N}_{[b],x}$ is the Newton stratum for $[b]$ in $IxI$.

\begin{cor}\label{coradlvnewton1}
For $x \in \widetilde{W}$ and $b \in IxI$, let $h\in \mathcal{N}_{[b],x}(k)$. Let $g\in X_x(b)(k)$ with $g^{-1}b\sigma(g)=h$. Denote by $X_x(b)^{\wedge}_g$ and $(\mathcal{N}_{[b],x})^{\wedge}_h$ the completions in the two points, respectively. Assume that $((\mathcal{N}_{[b],x})^{\wedge}_h)'$ is irreducible. Then 
\begin{equation*}
 \dim (X_x(b)^{\wedge}_g)' = \ell(x) - \langle 2 \rho, \nu(b) \rangle - \codim((\mathcal{N}_{[b],x})^{\wedge}_h)'.
\end{equation*}
\end{cor}
\begin{proof}
This follows from the previous theorem, using that $I_{\bar{N}}/ y_b^{-1}I_{\sigma^{-1}(\bar{N})}y_b$ is irreducible and of dimension $\ell(y_b)=\langle 2 \rho, \nu(b) \rangle.$
\end{proof}
\begin{cor}\label{coradlvnewton}
Using the notation of the previous corollary,
\begin{equation*}
 \dim X_x(b) = \ell(x) - \langle 2 \rho, \nu(b) \rangle - \codim(\mathcal{N}_{[b],x}),
\end{equation*}
 where $\codim(\mathcal{N}_{[b],x})$ denotes the minimal codimension of all irreducible components. Furthermore, $X_x(b)$ is equidimensional if and only if the same holds for $\mathcal{N}_{[b],x}$.
\end{cor}
\begin{proof}
Apply the previous corollary to all elements contained in just one irreducible component of $\mathcal{N}_{[b],x}$.
\end{proof}
Again, the dimension formula is known from \cite[Thm.~2.23]{He-CDM}, and the equidimensionality assertion is new. We are now able to prove Lemma \ref{lemdimgen} as an immediate consequence.

\begin{proof}[Proof of Lemma \ref{lemdimgen}]
Apply Corollary \ref{coradlvnewton} to $[b_x]$, and use that in this case the Newton stratum is irreducible and of codimension 0 in $IxI$.
\end{proof}

\subsection{Virtual dimension and cordiality}

We recall from \cite[Sec.~10.1]{HeAnnals} the notion of virtual dimension. For $x\in \widetilde W$ and $[b]\in B(G)$ with $\kappa_G(b)=\kappa_G(x)$, define
$$d_x(b)=\frac{1}{2}\left(\ell(x)+\ell(\eta(x))-\defect(b)-\langle 2\rho, \nu(b) \rangle\right)$$ to be the virtual dimension of the pair $(x,[b])$. By \cite[Thm.~2.30]{He-CDM}, we have 
for $x\in \widetilde W$ and $[b]\in B(G)$ with $\kappa_G(x)=\kappa_G(b)$ that 
\begin{equation}\label{lemhe104} \dim X_x(b)\leq d_x(b).
\end{equation} 

Notice that the additional assumptions on $G$ made in loc.~cit.~are not needed for the proof of the above theorem, so that it also holds in our context. We combine this with the formula for $\dim X_x(b)$ from the preceding subsection.

\begin{lemma}
Let $x\in \widetilde W$, and let $[b_x]\in B(G)$ be the generic $\sigma$-conjugacy class in $IxI$. Then 
\begin{equation*}
\ell(x) - \ell(\eta(x)) \leq \langle 2\rho, \nu_x \rangle - \defect (b_x).
\end{equation*}
\end{lemma}

\begin{proof}
By \eqref{lemhe104}, we have $\dim X_x(b_x)\leq d_x(b_x)$. Together with Lemma \ref{lemdimgen}, we have 
$$\ell(x)-\langle 2\rho, \nu_x \rangle\leq  \frac{1}{2}\left(\ell(x)+\ell(\eta(x))-\defect(b)-\langle 2\rho, \nu_x \rangle\right),$$
which is equivalent to the above inequality.
\end{proof}

\begin{definition}\label{defcd}
Let $x\in \widetilde W$. Let $[b_x]\in B(G)$ be the generic $\sigma$-conjugacy class in $IxI$. Then $x$ is called \emph{\cd }~if 
\begin{equation*}
\ell(x) - \ell(\eta(x)) = \langle 2\rho, \nu_x \rangle - \defect (b_x).
\end{equation*}
In other words, $x$ is \cd ~if and only if $\dim X_x(b_x)=d_x(b_x)$.
\end{definition}

\begin{ex}\label{exantidom}
Suppose that $x = t^{w_0\lambda}w \in \widetilde{W}$ so that $x$ is in the antidominant Weyl chamber. Then by Mazur's inequality, $\nu_x\leq \lambda$. Since $t^{w_0\lambda}\in \overline{IxI}$, the generic Newton point $\nu_x$ cannot be strictly smaller than $\lambda$, hence $[b_x]=[t^{\lambda}]$. Thus $\defect (b_x)=0$, and $\langle 2\rho, \nu_x \rangle=\ell(t^{\lambda})=\ell(x)-\ell(w_0w\sigma(w_0))=\ell(x)-\ell(\eta(x))$. Hence all $x$ in the antidominant Weyl chamber are cordial, which proves the first assertion (a) of Theorem \ref{T:MainEx}.
\end{ex} 

The following theorem is a stronger version of Theorem \ref{thmmaincd}. The idea of the proof (first used in \cite{grothconj}) is to combine a strong version of purity of the Newton stratification with upper bounds on the dimension of the Newton strata obtained via Corollary \ref{coradlvnewton}; see \cite[Sec.~5]{newtonsurv} for an overview. 

\begin{thm}\label{thm1strong}
Let $x\in \widetilde W$. Then for $[b]\in B(G)_x$ the following are equivalent.
\begin{enumerate}
\item $d_x(b')-\dim X_x(b')\geq d_x(b_x)-\dim X_x(b_x)$ for all $[b']\in B(G)_x$ with $[b]\leq[b']$.\\
\item $d_x(b')-\dim X_x(b')= d_x(b_x)-\dim X_x(b_x)$ for all $[b']\in B(G)_x$ with $[b]\leq[b']$.\\
\item If $[b']\in B(G)$ with $[b]\leq [b']\leq [b_x]$ then $[b']\in B(G)_x$ and the closure of $\mathcal N_{[b']}$ is the union of all $\mathcal N_{[b'']}$ for $[b'']\in B(G)_x$ with $[b'']\leq[b']$.
\end{enumerate}
If this is the case, then all $X_x(b')$ with $[b]\leq [b']$ are also equidimensional.
\end{thm}

\begin{proof}
We first reformulate (1). For $g\in X_x(b')$ let $b'_g=g^{-1}b'\sigma(g)$. Then (1) is equivalent to the following estimate for the dimension of the completion of $X_x(b')$ for each $[b']$ and each element $g$ as above.
\begin{equation*}
\dim X_x(b')^{\wedge}_g\leq \dim X_x(b')\leq d_x(b')-d_x(b_x)+\dim X_x(b_x).
\end{equation*} 
By Corollary \ref{coradlvnewton1} and Lemma \ref{lemdimgen} this is equivalent to
\begin{align}
\nonumber\codim (\mathcal N_{[b'],x})^{\wedge}_{b_g}&\geq d_x(b_x)-d_x(b')+\langle 2\rho,\nu_x-\nu(b')\rangle\\ 
\nonumber&= \frac{1}{2}\left(\defect(b')-\defect(b_x)+\langle 2\rho,\nu_x- \nu(b') \rangle\right),
\end{align}
where the last equality follows from the definition of the virtual dimension. 

The right hand side of this estimate is equal to the length of every maximal chain of elements in $B(G)$ from $[b']$ to $[b_x]$. Indeed, this follows directly from \cite[Thm.~3.4]{newtonsurv}, which in its turn is a slight correction of \cite[Thm.~7.4]{Ch}, combined with the main result of \cite{KotNewtStrata} and \cite[Prop.~3.8]{Hamacher_Newt}.

To show that (1) implies (2) and (3) as well as equidimensionality, we use \cite[Lem.~5.12]{newtonsurv}. The assumption on strong purity of the stratification made in the lemma can be replaced by topological strong purity in the sense of \cite[Sec.~2.1]{Ham_APS}. By \cite[Prop.~1]{Ham_APS}, this is satisfied for the stratification we consider. By the above reformulation of (1), the assumption of \cite[Lem.~5.12]{newtonsurv} on the codimensions of the Newton strata is satisfied. Then the lemma implies that the Newton stratification on the scheme associated with the completion of $IxI$ in any $b'_g$ satisfies that each Newton stratum associated with some $[\tilde b]\in B(G)$ with $[b']\leq [\tilde b]\leq [b_x]$ is non-empty, and its closure is the union of all Newton strata for $[b'']$ with $[b'']\leq[\tilde b]$. Since all of the above holds for every $g\in X_x(b')$, and in particular for all elements contained in exactly one irreducible component, (3) and equidimensionality follow.

It remains to show that (3) implies (1). Let $[b']\in B(G)$ with $[b]\leq [b'].$ Every chain $[b']=[b_0]<[b_1]<\dotsm<[b_n]=[b_x]$ in $B(G)$ is by (3) also a chain in $B(G)_x$. By the second assumption of (3) we have that $\mathcal N_{[b_i]}\subset \overline{\mathcal N_{[b_{i+1}]}}$ for all $i$. Thus for every $h\in \mathcal{N}_{[b']}$, the codimension of $(\mathcal N_{[b'],x})^{\wedge}_{h}$  is greater or equal to the maximal length of such a chain. By the above reformulation, this is equivalent to (1).
\end{proof}

We are now prepared to prove our first main result and an immediate corollary.

\begin{proof}[Proof of Theorem \ref{thmmaincd}]
By (\ref{lemhe104}) together with cordiality, we have $d_x(b')-\dim X_x(b')\geq 0= d_x(b_x)-\dim X_x(b_x)$ for all $[b]\leq[b']\in B(G)_x$. Then the theorem follows from Theorem \ref{thm1strong}.
\end{proof}
\begin{cor}
Let $x$ be \cd . Then for every $[b]\in B(G)_x$ we have that $X_x(b)$ is equidimensional of dimension $\dim X_x(b)=d_x(b)$.
\end{cor}

\begin{remark}\label{rembasic}
 If $x$ is in the shrunken Weyl chamber and the basic locus is non-empty, then \cite[Thm.~4.2]{HeSurvey} says that $\dim X_x(b)=d_x(b)$ for the basic class $[b]\in B(G)_x$. A necessary and sufficient criterion for non-emptiness of the basic locus is given in \cite{GoertzHeNie}. In this case, our theorem shows that if $x$ is \cd , then $B(G)_x=\{[b]\leq [b_x]\}$.
\end{remark}

Theorem \ref{thm1strong} also implies the following result which is a partial converse to Theorem \ref{thmmaincd}. 
\begin{cor}\label{thmconv}
Suppose that $x\in\widetilde W$ is not \cd . Assume that there is a $[b]\in B(G)_x$ such that $\dim X_x(b)=d_x(b)$. Then there is a $[b'] \in B(G)$ such that 
\begin{enumerate}
\item[(a)] $[b]< [b']< [b_x]$ but $[b']\notin B(G)_x$ \emph{(}in particular, $B(G)_x$ is not saturated\emph{)}, or
\item[(b)] $[b]< [b']< [b_x]$ and $[b']\in B(G)_x$, but the closure of $\mathcal N_{[b']}$ is not the union of all $\mathcal N_{[b'']}$ for $[b'']\in B(G)_x$ with $[b'']< [b']$.
\end{enumerate}
\end{cor}
\begin{proof}
We have $\dim X_x(b_x)<d_x(b_x)$, and hence $d_x(b)-\dim X_x(b)=0< d_x(b_x)-\dim X_x(b_x)$.
\end{proof}
Along these same lines, one could also formulate more precise statements relating $d_{x}(b_x)-\dim X_{x}(b_x)-d_x(b)+\dim X_x(b)$ to the number of $[b']$ as in Corollary \ref{thmconv}.

\section{Families of \cd~elements}\label{sec3}

Characterizing the \cd\ elements in $\widetilde{W}$ requires a good description of the maximal Newton point $\nu_x$. One especially useful description of $\nu_x$ uses paths in the quantum Bruhat graph and is available for groups $G$ which are split, connected, and semisimple. Thus for the remainder of the paper we make these additional assumptions on $G$.

Let $\Phi$ be the set of relative roots of $G$ over $\breve F$ with respect to $T$, and let $\Phi^+$ be the set of positive roots. Let $S$ be the basis of $\Phi$ of simple roots corresponding to $B$. We also identify ${S}$ with the set of simple reflections in $W$. The finite Weyl group $W$ acts on $\R^r$ as a finite reflection group, where $r$ is the rank of $G$. The set of reflections in $W$ is defined as $R = \{wsw^{-1} \mid s \in {S}, w \in W \}.$  There is a bijection  between $\Phi^+$ and $R$. More precisely, let $\alpha \in \Phi^+$ and write $\alpha=w(\alpha_i)$ for some simple root $\alpha_i$ and $w\in W$. Then $\alpha$ corresponds to the well-defined reflection $s_\alpha:=ws_iw^{-1} \in W$. Throughout, we denote simple reflections by $s_i$ (the index being a roman letter), and reflections associated with a positive root (which may or may not be simple) by $s_{\alpha}$ (the index being a greek letter). 

\subsection{\Cd~elements and the quantum Bruhat graph}

The primary tool in the proof of Theorem \ref{T:MainEx} (b) and (c) is a labeled directed graph associated with the group $G$ called the quantum Bruhat graph. We now review some key properties of this graph and its relation to maximal Newton points.

\begin{definition}[\cite{FGP}] We construct the \textit{quantum Bruhat graph} $\Gamma_G$ as follows.
\begin{enumerate}
\item The vertices of the graph are the elements $w \in W$. 
\item Draw a directed edge $w \longrightarrow ws_{\alpha}$ for any $\alpha \in \Phi^+$ if either of the following is satisfied:
\begin{align*}w\ \textcolor{blue}{\longrightarrow}\ ws_{\alpha} \quad \text{if}\  \ & \ell(ws_{\alpha}) = \ell(w)+1,\ \ \text{or}  \\
w\ \textcolor{red}{\longrightarrow}\ ws_{\alpha} \quad \text{if}\ \ & \ell(ws_{\alpha}) = \ell(w)-\langle  2\rho, \alpha^{\vee} \rangle +1.
\end{align*}
\item Label the edge $w \longrightarrow ws_{\alpha}$ by the corresponding root $\alpha$. 
\end{enumerate} 
\end{definition}

Figure \ref{fig:S_3QBG} shows the quantum Bruhat graph for $G={\rm SL}_3$.  As in Figure \ref{fig:S_3QBG}, we can always draw $\Gamma_G$ such that vertices are ranked by length increasing upward, in which case the first type of edge (colored blue) always points upward and the second type (colored red) downward; this will be our convention throughout the paper. 
 Note that the upward edges correspond precisely to the covering relations in Bruhat order, and so we can also view the vertices in $\Gamma_G$ as a ranked partially ordered set. We write $v \lessdot w$  if $v \leq w$ in Bruhat order and $\ell(v) = \ell(w)-1$ to denote such a covering relation. 

 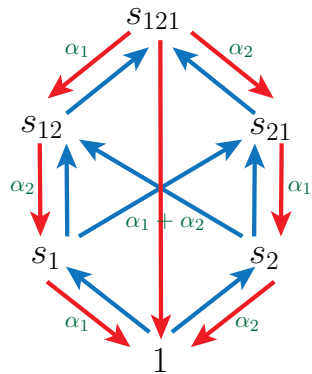
\begin{figure}[h]
\begin{center}
 \resizebox{1.9in}{!}
{
\begin{overpic}{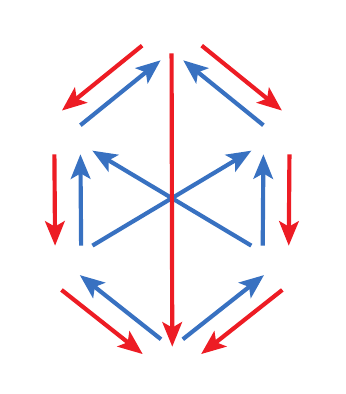}
\put(65.5,63){\bf \huge{$s_{21}$}}
\put(9,63.5){\bf \huge{$s_{12}$}}
\put(65.5,31){\bf \huge{$s_{2}$}}
\put(11.5,31){\bf \huge{$s_{1}$}}
\put(41.5,4){\huge{1}}
\put(35,90.5){\bf \huge{$s_{121}$}}
\put(20,15){\bf \large{\textcolor{dkgreen}{$\alpha_1$}}}
\put(62,15){\bf \large{\textcolor{dkgreen}{$\alpha_2$}}}
\put(75,49){\bf \large{\textcolor{dkgreen}{$\alpha_1$}}}
\put(6.5,49){\bf \large{\textcolor{dkgreen}{$\alpha_2$}}}
\put(20,83){\bf \large{\textcolor{dkgreen}{$\alpha_1$}}}
\put(60.5,83){\bf \large{\textcolor{dkgreen}{$\alpha_2$}}}
\put(35,40){\bf \large{\textcolor{dkgreen}{$\alpha_1$}}}
\put(48.5,40){\bf \large{\textcolor{dkgreen}{$\alpha_2$}}}
\put(42.5,40){\bf \large{\textcolor{dkgreen}{$+$}}}
\end{overpic}
}
\caption{The quantum Bruhat graph $\Gamma_G$ for  $G={\rm SL}_3$.}\label{fig:S_3QBG}
\end{center}
\end{figure}

 Define the \emph{weight of an edge} in the quantum Bruhat graph $\Gamma_G$ as follows.
\begin{enumerate}
\item An upward edge $w\ \textcolor{blue}{\longrightarrow}\ ws_{\alpha}$ carries no weight.  
\item A downward edge $w\ \textcolor{red}{\longrightarrow}\ ws_{\alpha}$ carries a weight of $\alpha^{\vee}$.
\end{enumerate} 
The \emph{weight of a path} in $\Gamma_G$ is the sum of the weights over all of the edges in the path.  For example, in $\Gamma_{{\rm SL}_3}$ from Figure \ref{fig:S_3QBG}, the weight of each of the three shortest paths from $s_1s_2 = s_{12}$ to $s_2$ equals $\alpha^{\vee}_1+\alpha^{\vee}_2$.  In general, given any $u,v \in W$, by \cite[Lem.~1]{PostQBG} there always exists a path in $\Gamma_G$ from $u$ to $v$, and all paths of minimal length between $u$ and $v$ have the same weight.

Since $G$ is split, connected, and semisimple, then under a superregularity hypothesis guaranteeing that $x = t^{v\lambda}w \in \widetilde{W}$ is sufficiently far from the walls of any Weyl chamber, the maximal Newton point $\nu_x$ can be computed from the weight of certain paths in the quantum Bruhat graph $\Gamma_G$. More specifically, \cite[Thm.~3.2]{BeMaxNPs} says that under a superregularity hypothesis on $\lambda$, the maximal Newton point $\nu_x$ can be expressed as
 \begin{equation}\label{E:Nu_xQBG}
 \nu_x = \lambda - \alpha^{\vee}_x,
 \end{equation}
where $\alpha^{\vee}_x$ denotes the weight of any path of minimal length from $w^{-1}v$ to $v$ in $\Gamma_G$. 

Denote by $d_{\Gamma}(u,v)$ the minimum length among all paths in $\Gamma_G$ from $u$ to $v$; the choice of notation represents the fact that $d_{\Gamma}(u,v)$ equals the distance between these two elements in the graph $\Gamma_G$. As an important special case, denote the minimum length of any path in $\Gamma_G$ from $w$ to the identity which uses exclusively downward edges by $d_{\downarrow}(w)$.  We remark that such a path always exists, since by definition of $\Gamma_G$ any reduced expression for $w$ determines an all downward path from $w$ to the identity having $\ell(w)$ edges.  We say that any path in $\Gamma_G$ from $u$ to $v$ which uses exactly $d_\Gamma(u,v)$ edges \emph{realizes} $d_\Gamma(u,v)$.  Similarly, any downward path in $\Gamma_G$ from $w$ to 1 consisting of exactly $d_{\downarrow}(w)$ edges \emph{realizes} $d_{\downarrow}(w)$.

We are now able to characterize the \cd~elements under our additional superregularity hypothesis in a purely combinatorial manner which does not require any explicit knowledge of the maximal Newton point.

\begin{prop}\label{etaminpath}
Let $x = t^{v\lambda}w \in \widetilde{W}$, and suppose that $\langle \alpha_i, \lambda \rangle > M$ for all simple roots $\alpha_i$, where $M$ is the constant defined in \cite[Eq.~6.1]{BeMaxNPs}.
Then $x$ is \cd\ if and only if 
\begin{equation*}
d_{\Gamma}(w^{-1}v, v) = \ell(v^{-1}wv) = \ell(\eta(x)).
\end{equation*}
\end{prop}

\begin{proof}
First note by \eqref{E:Nu_xQBG} that $\nu_x$ is integral under our superregularity hypothesis on $\lambda$.  Therefore, $\defect (b_x) = 0$ in this case, and so $x$ is \cd\ if and only if
$\ell(x) - \ell(\eta(x)) = \langle 2\rho, \nu_x \rangle.$
Now recall a length formula for $x$ from \cite[Lem.~3.4]{LSPeterson}, which applies since $\lambda$ is both regular and dominant:
\begin{equation}\label{E:xlengthreform}
\ell(x) = \ell(t^{\lambda}) - \ell(w^{-1}v) + \ell(v) = \langle 2\rho, \lambda \rangle - \ell(w^{-1}v)+\ell(v).
\end{equation}
Combine Equations \eqref{E:Nu_xQBG} and \eqref{E:xlengthreform} to write
\begin{align*}
\ell(x) - \langle 2\rho, \nu_x \rangle 
& =  \bigl( \langle 2\rho, \lambda \rangle - \ell(w^{-1}v)+\ell(v)\bigr) - \langle 2\rho, \lambda - \alpha^{\vee}_x \rangle\\
& = \langle 2\rho, \alpha^{\vee}_x \rangle - \ell(w^{-1}v) + \ell(v),
\end{align*}
where $\alpha^\vee_x$ is the weight of any minimal length path $p$ in $\Gamma_G$ from $w^{-1}v$ to $v$.   Therefore, $x$ is \cd\ if and only if 
$\langle 2 \rho, \alpha^\vee_x \rangle - \ell(w^{-1}v) + \ell(v) = \ell(\eta(x)).$
It thus suffices to show that 
\begin{equation}\label{suffices}
\langle 2 \rho, \alpha^\vee_x \rangle - \ell(w^{-1}v) + \ell(v) = d_\Gamma(w^{-1}v,v).
\end{equation} 

Note that the quantity $-\ell(w^{-1}v)+\ell(v)$ equals the difference in rank in the poset $\Gamma_G$ from the beginning to the end of the path $p$, where the quantity is positive, negative, or zero according to whether the rank of the final vertex of $p$ is higher, lower, or the same as the rank of its initial vertex.  For ease of reference, denote this quantity by $\Delta \rk(p) = -\ell(w^{-1}v)+\ell(v)$.  
Recall that we draw an edge $w \longrightarrow ws_\alpha$ in $\Gamma_G$ if and only if 
\begin{equation*}
\ell(ws_{\alpha}) = 
\begin{cases}
\ell(w)+1,\ \  \text{or}  \\
\ell(w)-\langle 2\rho, \alpha^{\vee} \rangle +1,
\end{cases}
\end{equation*}
where the edges of the first type are directed upward and the second type are directed downward.  Therefore, each upward edge in $p$ contributes $+1$ to $\Delta \rk(p)$, and each downward edge in $p$ labeled by $\alpha$ contributes $-\langle 2\rho, \alpha^\vee\rangle +1$ to $\Delta \rk(p)$.  Denote the roots labeling the downward edges by $\alpha_{d_i}$ for $i= 1, \dots, d$ where $d$ equals the number of downward edges in $p$. Denote the number of upward edges in $p$ by $u$. We can thus write 
\begin{align}\label{rkrewrite}
\Delta \rk(p) & = u +  \sum\limits_{i=i}^d (-\langle 2\rho, \alpha^\vee_{d_i} \rangle + 1) \nonumber \\
& = (u+d)  -  \sum\limits_{i=i}^d \langle 2\rho, \alpha^\vee_{d_i} \rangle  \\
&= d_\Gamma(w^{-1}v,v) - \sum\limits_{i=i}^d \langle 2\rho, \alpha^\vee_{d_i} \rangle. \nonumber
\end{align}
On the other hand, recall that the weight of the path $p$ is defined to be $\sum \alpha_{d_i}^\vee$ summing over all the downward edges, so that by linearity we can rewrite \eqref{rkrewrite} as
\begin{equation*}
\Delta \rk(p) = d_\Gamma(w^{-1}v,v) - \left\langle 2\rho, \sum\limits_{i=i}^d  \alpha^\vee_{d_i} \right\rangle = d_\Gamma(w^{-1}v,v) - \langle 2\rho, \alpha^\vee_x\rangle.
\end{equation*}
Therefore,
\begin{equation*}
d_\Gamma(w^{-1}v,v) = \langle 2\rho, \alpha^\vee_x \rangle + \Delta \rk(p)  = \langle 2 \rho, \alpha^\vee_x \rangle - \ell(w^{-1}v) + \ell(v),
\end{equation*}
confirming \eqref{suffices} and concluding the proof.
\end{proof}

\begin{lemma}\label{ineqrmk}
For any $u,v\in W$ we have $d_{\Gamma}(u, v) \leq \ell(u^{-1}v)$.
\end{lemma}
\begin{proof}
Taking any reduced expression for $u^{-1}v = s_{i_1}\cdots s_{i_k}$ and following the edges labeled by the simple roots $\alpha_{i_1}, \dots, \alpha_{i_k}$ in order, we obtain a path from $u$ to $v$ which has exactly $\ell(u^{-1}v)$ edges.  
\end{proof}

\begin{remark} In particular, under the superregularity hypothesis, \cd~elements are precisely those for which no shorter path exists from $w^{-1}v$ to $v$ than the one constructed in the proof of Lemma \ref{ineqrmk}, where $v,w$ are as in Proposition \ref{etaminpath}.
\end{remark}

We now provide an example which illustrates how to use Proposition \ref{etaminpath} to identify families of cordial elements. Recall that we already considered this case (in greater generality) in Example \ref{exantidom}. 

\begin{ex}\label{antidom}
Suppose that $x = t^{w_0\lambda}w \in \widetilde{W}$ so that $x$ is in the antidominant Weyl chamber.  If $\lambda$ is superregular, we want to show also with this new method that $x$ is cordial. By Proposition \ref{etaminpath}, it suffices to prove that $d_\Gamma(w^{-1}w_0,w_0) = \ell(\eta(x))$.  Since the end vertex of the path in $\Gamma_G$ is the longest element $w_0$, and since every upward edge only increases the length by one, any path of minimal length ending at $w_0$ is necessarily a path containing only upward edges.  Comparing rank, any minimal path from $w^{-1}w_0$ to $w_0$ then has exactly $\ell(w_0) - \ell(w^{-1}w_0) = \ell(w_0)- \ell(w_0w)$ edges.  Recall from \cite[Cor.~2.3.3]{BB} that $\ell(w_0w) = \ell(w_0)-\ell(w)$ and $\ell(w_0ww_0) = \ell(w)$ for all $w \in W$.  Therefore, for these elements, we have 
\[ \ell(\eta(x)) = \ell(w_0ww_0) = \ell(w) = \ell(w_0) - \ell(w^{-1}w_0) = d_\Gamma(w^{-1}w_0,w_0).\] By Proposition \ref{etaminpath}, $x$ is cordial.  Compare Theorem \ref{T:MainEx} (a), which we recall was proved in Example \ref{exantidom}, without any superregularity hypothesis.
\end{ex}

\subsection{Standard parabolic Coxeter and \sha~elements}\label{spcsha}

In this section, we develop the necessary background to study the latter two families of cordial elements identified in Theorem \ref{T:MainEx}.

The \emph{reflection length} of $w \in W$ is the minimal number of reflections required to express $w$ as a product of elements in $R$; namely,
\begin{equation*}
\ell_R(w) = \min \left\{ r \in \mathbb{N} \ \middle|\ w = s_{\beta_1}\cdots s_{\beta_r}\ \text{for}\ s_{\beta_i} \in R \right\}.
\end{equation*}
By definition, $\ell_R(w) \leq \ell(w)$. We now recall a characterization of those elements such that $\ell_R(w) = \ell(w)$.

\begin{definition}\label{parcoxeter}
The element $w \in W$ is a \emph{standard parabolic Coxeter element} if each simple reflection is used at most once in any (equivalently every) reduced expression for $w$.  
\end{definition}

As the terminology suggests, standard parabolic Coxeter elements are those which are Coxeter elements in some standard parabolic subgroup of $W$. (We remark that standard parabolic Coxeter elements have also appeared by other names in the literature; for example, they are called \emph{boolean} in \cite{RagTenner}.)  By \cite[Lem.~2.1]{BDSW}, the element $w$ is standard parabolic Coxeter if and only if $\ell_R(w) = \ell(w)$, a property which will be critical in the proof of Theorem \ref{T:MainEx} (b).

Next we define a slightly more general family of elements in $W$, which properly contains the standard parabolic Coxeter elements. 

\begin{definition}\label{sha} We say $w \in W$ \emph{contains} the element $v \in W$ if there exist $u,u'\in W$ such that $w = uvu'$ and $\ell(w) = \ell(u)+\ell(v)+\ell(u')$. An element $w \in W$ is called  \emph{small-height-containing} if $w$ contains a non-simple reflection $s_\alpha$ such that $\ell(s_\alpha)=\langle 2\rho, \alpha^\vee \rangle -1$. Otherwise, we say that $w$ is \emph{\sha}.
\end{definition}
\noindent Note that all simple reflections $\alpha_i$ satisfy $\ell(s_{\alpha_i}) = \langle 2\rho, \alpha_i\rangle -1$, so we intentionally exclude these. Also note that the \sha\ condition cannot be verified by looking at only one reduced expression, as the example $s_{1213} = s_{1231}$ in type $A_3$ illustrates.

This terminology is inspired by the related notion of \emph{short-braid-avoiding} elements, which are those elements of $W$ which do not contain a subexpression of the form $s_is_js_i$ in any reduced expression; see \cite{Fan}. If $G$ is simply-laced, then for any $\alpha \in \Phi^+$ we have $\ell(s_\alpha) = \langle 2\rho, \alpha^\vee \rangle -1$ by \cite[Lem.~4.3]{BFP}, and so the notions of \sha~and short-braid-avoiding coincide in this case.  More generally, for any $\alpha \in \Phi^+$ we always have $\ell(s_\alpha) \leq \langle 2\rho, \alpha^\vee \rangle -1$, and the inequality may be strict.  Rewriting this expression, we see that $\height \alpha^\vee \geq \frac{\ell(s_\alpha)+1}{2},$ and so those reflections which we avoid in Definition \ref{sha} are precisely those whose height is as ``small'' as it could possibly be.  There is also a relationship between the \sha~and \emph{fully commutative} elements defined in \cite{StFC}, which are those for which any reduced expression can be obtained from any other by means of only commuting relations.  In the simply-laced case,  it follows from \cite[Prop.~2.1]{StFC} that all of these notions coincide. 

\begin{ex}
As an example which illustrates the relations among these families, we identify the standard parabolic Coxeter, \sha, short-braid-avoiding, and fully commutative elements for $G$ of type $C_2$.  In this case, $W = \langle s_1, s_2 \mid s_1^2=s_2^2=(s_1s_2)^4=1\rangle$ so that the four reflections are $s_1, s_2, s_{121},$ and $s_{212}$, and the other nontrivial elements (all of which are rotations in $\R^2$) are $s_{12}, s_{21}$, and $w_0=s_{1212}$.  The standard parabolic Coxeter elements are thus $\{1, s_1, s_2, s_{12}, s_{21}\},$ which coincides here with the set of short-braid-avoiding elements.  All of the elements besides $w_0$ are fully commutative.  To determine the \sha~elements, we must further identify the coroots which correspond to each non-simple reflection. We follow the convention that $\alpha_1$ is the short simple root and $\alpha_2$ the long one. Then
\begin{eqnarray*}
s_{121}  \quad \longleftrightarrow & \alpha_1+\alpha_2 & \longleftrightarrow \quad \alpha_1^\vee + \alpha_2^\vee, \\
s_{212}  \quad \longleftrightarrow & 2\alpha_1+\alpha_2 & \longleftrightarrow \quad \alpha_1^\vee + 2\alpha_2^\vee.
\end{eqnarray*}
We thus see that $\ell(s_{121})  =  \langle 2\rho, \alpha_1^\vee + \alpha_2^\vee \rangle -1$, so that \sha~elements cannot contain $s_{121}$.  By contrast, $3=\ell(s_{212})  \neq  \langle 2\rho, \alpha_1^\vee + 2\alpha_2^\vee \rangle -1 = 5$, so $s_{212}$ does not need to be avoided.  Therefore, the set of \sha~elements in $C_2$ is $\{1, s_1, s_2, s_{12}, s_{21}, s_{212}\}$, which sits properly between the sets of standard parabolic Coxeter (or equivalently, short-braid-avoiding) and fully commutative elements.
\end{ex}

\subsection{Two additional families of \cd~elements}

The goal of this section is to prove parts (b) and (c) of Theorem \ref{T:MainEx}.  For part (c), we first require two more technical lemmas as stepping stones to Proposition \ref{mindown}, which allows us to focus exclusively on paths in $\Gamma_G$ with all downward edges.

\begin{lemma}\label{refprod}
Let $s_\beta \in R$ be a non-simple reflection such that $\ell(s_\beta) = \langle 2\rho, \beta^\vee \rangle -1$ for some $\beta \in \Phi^+$, and suppose that $s_\beta s_\alpha \lessdot s_\beta$ for some $\alpha \in \Phi^+$.  Then $s_\beta s_\alpha = s_{\gamma_1}s_{\gamma_2}$, where  $\ell(s_\beta s_\alpha) = \ell(s_{\gamma_1}) + \ell(s_{\gamma_2})$ and $\ell(s_{\gamma_i}) = \langle 2\rho, \gamma^\vee_i \rangle -1$.
\end{lemma}

\begin{proof}
For any reduced expression $s_\beta =  s_{i_1} \cdots s_{i_m}$, the condition $s_{\beta}s_{\alpha}\lessdot s_{\beta}$ together with the Strong Exchange Property implies that there is a reduced expression $s_\beta s_\alpha = s_{i_1} \cdots \widehat{s_{i_l}} \cdots s_{i_m}$ for some $1 \leq l \leq m$. Moreover, since $s_\beta$ is a reflection, $\ell(s_\beta) = m$ is odd, and we may choose the reduced expression for $s_\beta$ to be palindromic by \cite[Lem.~4.1]{BFP}. For $l=(m+1)/2$, the resulting expression for $s_{\beta}s_{\alpha}$ is trivial, and the hypothesis $s_\beta s_\alpha \lessdot s_\beta$ is not satisfied. Thus for symmetry reasons, it is enough to consider the cases where $l> (m+1)/2$. In this situation, we have $s_{\alpha}=s_{i_m}\cdots s_{i_l} \cdots s_{i_m}$, and $\ell(s_{\alpha})<\ell(s_{\beta})-1=\ell(s_{\beta}s_{\alpha})=\ell(s_{\alpha}s_{\beta})$ . By \cite[Prop.\ 4.4.6]{BB}, this inequality implies that $s_{\alpha}(\beta)>0$. By the same proposition, $\ell(s_{\beta}s_{\alpha}) < \ell(s_{\beta})$ implies that $s_{\beta}(\alpha)=\alpha-\langle \alpha,\beta^{\vee}\rangle\beta<0$. Since $\alpha$ and $\beta$ are positive, $\langle \alpha, \beta^{\vee}\rangle$ also has to be positive. Therefore, 
\begin{equation}\label{eqc'}
(s_{\alpha}(\beta))^{\vee}=s_{\alpha^{\vee}}(\beta^{\vee})=\beta^{\vee}-c'\alpha^{\vee}
\end{equation} for some integral $c'>0$.

Now, recalling that $l> (m+1)/2$, we will prove that we may choose $\gamma_1=\alpha$ and $\gamma_2=s_{\alpha}(\beta)$ to satisfy the conclusion of the lemma.  Certainly, $s_\beta s_\alpha = s_\alpha (s_\alpha s_\beta s_\alpha) = s_{\gamma_1}s_{\gamma_2}$.  We next show that this product is length-additive. Since $\ell(s_{\beta})-1 =\ell(s_{\beta}s_{\alpha}) =  \ell(s_{s_{\alpha}(\beta)}s_{\alpha}) \leq \ell(s_{s_{\alpha}(\beta)})+\ell(s_{\alpha})$, length-additivity is implied by the following claim.
\vskip 3pt
{\it Claim.} $\ell(s_{s_{\alpha}(\beta)})+\ell(s_{\alpha})\leq \ell(s_{\beta})-1.$
\vskip 3pt
For every positive root $\gamma$ we have $\ell(s_{\gamma})\leq\langle 2\rho,\gamma^{\vee}\rangle-1$ by \cite[Lem.~4.3]{BFP}, and we assumed equality for $\gamma=\beta$.  Then we have
\begin{eqnarray}\label{ineqs}
\ell(s_{s_{\alpha}(\beta)})&\leq & \langle 2\rho, (s_{\alpha}(\beta))^{\vee}\rangle-1 \nonumber\\
&\overset{\eqref{eqc'}}{=}&\langle 2\rho, \beta^{\vee}-c'\alpha^{\vee}\rangle-1\\
&\leq& \langle 2\rho, \beta^{\vee}-\alpha^{\vee}\rangle-1 \nonumber\\
&\leq&\ell(s_{\beta})-\ell(s_{\alpha})-1, \nonumber
\end{eqnarray}
which proves the claim. 

Furthermore, since both sides of the inequality in the claim are equal, each of the inequalities in \eqref{ineqs} also has to be an equality. 
From the first line of \eqref{ineqs}, we see that $\ell(s_{s_{\alpha}(\beta)})= \langle 2\rho, (s_{\alpha}(\beta))^{\vee}\rangle-1$.  Finally, since $\ell(s_{\beta})=\langle 2\rho,\beta^{\vee}\rangle-1$ by hypothesis, the last equality in \eqref{ineqs} yields $\ell(s_{\alpha})=\langle 2\rho,\alpha^{\vee}\rangle-1$, which completes the proof.
\end{proof}

Lemma \ref{refprod} comprises the technical heart of the proof of Lemma \ref{basecase}, which says that using an upward edge does not ultimately provide savings on the number of edges required to go from an element $w$ down to the identity in $\Gamma_G$. 

\begin{lemma}\label{basecase}
Let $w \in W$, and suppose that $w \lessdot ws_\alpha$ for some $\alpha \in \Phi^+$.  Then
\begin{equation*}
d_{\downarrow}(w) \leq d_{\downarrow}(ws_\alpha) + 1.
\end{equation*}
\end{lemma}

\begin{proof}
Let $w \in W$, and suppose that $w \lessdot ws_\alpha$ for some $\alpha \in \Phi^+$.  Consider any path in $\Gamma_G$ realizing $d_{\downarrow}(ws_\alpha)$, which then corresponds to a length-additive expression as a product of reflections of the form $ws_\alpha = s_{\beta_1} \cdots s_{\beta_r}$, where each of the reflections satisfies $\ell(s_{\beta_i}) = \langle 2\rho, \beta_i^\vee \rangle -1$.  

On the other hand, since $w \lessdot ws_\alpha$ is a cocover, then for any reduced expression $ws_\alpha = s_{i_1} \cdots s_{i_k}$, we have  $w = s_{i_1} \cdots \widehat{s_{i_\ell}} \cdots s_{i_k}$ for some $1 \leq \ell \leq k$ by the Strong Exchange Property.  Further, since $\ell(w) = \ell(ws_\alpha)-1$, then the expression $w = s_{i_1} \cdots \widehat{s_{i_\ell}} \cdots s_{i_k}$ is still reduced. Therefore, $w$ has a reduced expression of the form $w = s_{\beta_1} \cdots (s_{j_1}\cdots \widehat{s_{j_p}} \cdots s_{j_m}) \cdots s_{\beta_r}$, where $s_{i_\ell} = s_{j_p}$ is the single factor removed from the reflection $s_{\beta_j} =  s_{j_1}\cdots  s_{j_m}$.  Since the entire expression for $w$ remains reduced when removing $s_{j_p}$, then the expression $s_{j_1}\cdots \widehat{s_{j_p}} \cdots s_{j_m}$ is also reduced.  Defining $s_\gamma = s_{j_m} \cdots s_{j_p}\cdots s_{j_m}$, we then see that $s_{\beta_j}s_\gamma \lessdot s_{\beta_j}$, and the hypotheses of Lemma \ref{refprod} are satisfied.  Therefore, we may write $s_{\beta_j}s_\gamma =  s_{\gamma_1}s_{\gamma_2}$, where  $\ell(s_{\beta_j}s_\gamma) = \ell(s_{\gamma_1}) + \ell(s_{\gamma_2})$ and $\ell(s_{\gamma_i}) = \langle 2\rho, \gamma^\vee_i \rangle -1$.

Altogether, we have thus shown that we have a length-additive expression for $w$ as a product of reflections of the form $w = s_{\beta_1} \cdots s_{\beta_{j-1}}s_{\gamma_1}s_{\gamma_2}s_{\beta_{j+1}} \cdots s_{\beta_r}$, where each of the reflections in the product satisfies the criterion for drawing a downward edge in $\Gamma_G$.  Therefore, this expression corresponds to a downward path of length $r+1 = d_{\downarrow}(ws_\alpha) +1$ from $w$ to 1 in $\Gamma_G$, and so $d_{\downarrow}(w) \leq d_{\downarrow}(ws_\alpha) + 1.$
\end{proof}

Lemma \ref{basecase} provides the foundation for the proof of Proposition \ref{mindown}, which allows us to trade paths from $w$ to the identity containing upward edges for a path of the same length that uses exclusively downward edges.

\begin{prop}\label{mindown}
Let $w \in W$.  Then $d_\Gamma(w,1) = d_{\downarrow}(w).$
\end{prop}

\begin{proof}
Define $m$ to be the minimal number of upward edges contained in any path in $\Gamma_G$ realizing $d_\Gamma(w,1)$. We have to prove that $m=0$. Assume that $m \geq 1$, and let $p$ be such a path. Denote the upward edges in $p$ by $u_i \longrightarrow u_is_{\beta_i}$ encountered in the order $i = 1, \dots, m$ as we travel along the path. Consider the subpath of $p$ which starts at $u_m$. Since the edge $u_m \longrightarrow u_ms_{\beta_m}$ is upward, then the length only increases by one and $u_m\lessdot u_ms_{\beta_m}$. Lemma \ref{basecase} then says that $d_{\downarrow}(u_m) \leq d_{\downarrow}(u_ms_{\beta_m})+1$.  Therefore, the subpath of $p$ beginning at $u_m$, which continues upward to $u_ms_{\beta_m}$, contains at least as many edges as any path realizing $d_{\downarrow}(u_m)$.  Define a new path $p_m$ in $\Gamma_G$ from $w$ to 1 by following the original path $p$ until the vertex $u_m$, after which we follow any path down to $1$ realizing $d_{\downarrow}(u_m)$. By Lemma \ref{basecase} and the fact that $p$ realizes $d_\Gamma(w,1)$, the length of the path $p_m$ also equals $d_\Gamma(w,1)$. However, the path $p_m$ has $m-1$ upward edges, contradicting the minimality of $m$ and proving that indeed $m=0$.
\end{proof}

We are now prepared to complete the proof of Theorem \ref{T:MainEx}.

\begin{proof}[Proof of Theorem \ref{T:MainEx}]
Recall that part (a) was already proved in Example \ref{exantidom}, and so it remains only to prove parts (b) and (c).  Let $x = t^{v\lambda}w \in \widetilde{W}$, and suppose that $\lambda$ is superregular in the sense of Proposition \ref{etaminpath}, and thus also Theorem \ref{T:MainEx}.  Then $x$ is cordial if and only if $d_\Gamma(w^{-1}v,v) = \ell(v^{-1}wv)$.

(b) We first prove that if $\eta(x)= v^{-1}wv$ is a standard parabolic Coxeter element, then $x$ is cordial.  Consider any path which realizes $d_\Gamma(w^{-1}v,v)=m,$ say  
\begin{equation*}
w^{-1}v \longrightarrow w^{-1}vs_{\beta_1} \longrightarrow w^{-1}vs_{\beta_1}s_{\beta_2} \longrightarrow \cdots \longrightarrow w^{-1}vs_{\beta_1} \cdots s_{\beta_m} = v.
\end{equation*}
Note that $v^{-1}wv = s_{\beta_1}\cdots s_{\beta_m}$ so this path corresponds to an expression for $\eta(x)$ as a product of $m$ reflections.  By definition, $\ell_R(\eta(x)) \leq m$, but since $\eta(x)$ is standard parabolic Coxeter, by \cite[Lem.~2.1]{BDSW} we have
\begin{equation*}
\ell(\eta(x)) = \ell_R(\eta(x)) \leq m = d_\Gamma(w^{-1}v,v).
\end{equation*}
The opposite inequality follows from Lemma \ref{ineqrmk}.  Therefore, if $\eta(x)$ is a standard parabolic Coxeter element, we see that $d_\Gamma(w^{-1}v,v) =\ell(\eta(x))= \ell(v^{-1}wv)$, and so $x$ is cordial by Proposition \ref{etaminpath}.

(c) We now prove that if  $x$ is in the dominant Weyl chamber, then $x$ is cordial if and only if $\eta(x)=w$ is \sha.  Since $v=1$ when $x$ is dominant, by Proposition \ref{etaminpath}, we aim to prove that $d_{\Gamma}(w^{-1}, 1) = \ell(w)$ if and only if $w$ is \sha.  Note that $w$ is \sha~if and only if $w^{-1}$ is \sha, and recall that $\ell(w)=\ell(w^{-1})$.  Therefore, in fact it suffices to prove that $d_\Gamma(w,1) = \ell(w)$ if and only if $w$ is \sha. 

First suppose that $w$ is small-height-containing.  By definition, there exists an expression for $w$ of the form $w = us_\beta v$, where $\ell(w) = \ell(u)+\ell(s_\beta)+\ell(v)$, for some $u, v \in W$ and $s_\beta$ some non-simple reflection such that $\ell(s_\beta) = \langle 2\rho, \beta^\vee \rangle-1$.  Taking any reduced expressions for $u$ and $v$, say $u = s_{i_1}\cdots s_{i_k}$ and $v = s_{j_1}\cdots s_{j_\ell}$, we can construct the following path in $\Gamma_G$
\begin{equation*}
w \xrightarrow{\alpha_{j_\ell}} ws_{j_\ell} \xrightarrow{\alpha_{j_{\ell-1}}} \cdots \xrightarrow{\alpha_{j_1}} w s_{j_\ell} \cdots s_{j_1} \xrightarrow{\ \beta\ } w s_{j_\ell} \cdots s_{j_1}s_\beta \xrightarrow{\alpha_{i_k}} \cdots \xrightarrow{\alpha_{i_1}} 1.
\end{equation*}
Each edge exists because length is additive in the expression $w=us_\beta v$, which means that at each step in this path the length drops by precisely $1 = \ell(s_{i_m}) = \ell(s_{j_n})$ or $\ell(s_\beta) = \langle 2\rho, \beta^\vee \rangle-1$, as required for a downward edge in $\Gamma_G$.  Since $s_\beta$ is non-simple, then $\ell(s_\beta) \geq 3$, which means that the length of this particular path is at most $\ell(w)-2$.  Therefore, $d_\Gamma(w,1) \leq \ell(w)-2 < \ell(w)$ in this case, and so $x$ is not cordial by Proposition \ref{etaminpath}.

Conversely, assume that $w$ is \sha. We aim to show that $d_\Gamma(w,1) = \ell(w)$. Recall Proposition \ref{mindown}, which says that $d_\Gamma(w,1) = d_{\downarrow}(w),$ and so there exists a path $p$ consisting of all downward edges which also minimizes length among all paths from $w$ to 1.  By the definition of the downward edges in $\Gamma_G$, this path corresponds to an expression $w = s_{\beta_1}\cdots s_{\beta_r}$ such that the length decreases by exactly $\langle 2\rho, \beta_i^\vee \rangle -1$ for all $1 \leq i \leq r$ when right multiplying $w$ by $s_{\beta_r}, \dots, s_{\beta_1}$ in  order. Note, however, that length cannot decrease by more than $\ell(s_{\beta_i})$ when right multiplying by $s_{\beta_i}$.  On the other hand, we always have $\ell(s_{\beta_i}) \leq \langle 2\rho, \beta_i^\vee \rangle -1$, and so in fact $\ell(s_{\beta_i}) = \langle 2\rho, \beta_i^\vee \rangle -1$ for all $1 \leq i \leq r$.  Therefore, the expression $w = s_{\beta_1}\cdots s_{\beta_r}$ is also length-additive. By definition of \sha, $w$ cannot contain any non-simple reflection $s_\beta$ such that $\ell(s_\beta) = \langle 2\rho, \beta^\vee \rangle-1$.  This means that each reflection in the expression $w = s_{\beta_1}\cdots s_{\beta_r}$ must in fact be simple, and so $\ell(w) = d_{\downarrow}(w) = d_\Gamma(w,1)$.  The element $x$ is thus cordial by Proposition \ref{etaminpath}.
\end{proof}

\begin{ex}\label{SL3Ex} 
For $G={\rm SL}_3$, the Newton stratification of each double coset $IxI$ has been computed in  \cite{BeThesis}. Note, however, that our description below corrects an error in the tables at the end of loc.\ cit. In ${\rm SL}_3$, all Newton strata are equidimensional, and the closure of any Newton stratum $[b]\cap IxI\neq \emptyset$ in $IxI$ is equal to the union of all $[b']\cap IxI$ such that $[b']\in B(G)_x$ and $[b']\leq [b]$. Write $x=t^{v\lambda}w$, and first assume that $v=1$,~i.e. $x$ is in the dominant Weyl chamber, and that $\lambda=(\lambda_1,\lambda_2,\lambda_3)$ with $|\lambda_i-\lambda_{i+1}|\neq 1$. Then $x$ is non-cordial if and only if $w=w_0$. Thus in this case, we obtain exactly the condition of Theorem \ref{T:MainEx} (b) or equivalently (c), but under a much weaker superregularity assumption on $\lambda$. Furthermore, all non-cordial elements (even without any regularity assumption) are of the form $x\omega$ for some non-cordial $x$ in the dominant Weyl chamber and $\omega$ normalizing $I$. For $x$ outside the dominant Weyl chamber with $v \in \{s_1, s_2, w_0\}$, there exist cordial elements which are not covered by Theorem \ref{T:MainEx} applied directly to $x$ or to $x\omega$ for any $\omega$ normalizing $I$. 
\end{ex}


\bibliographystyle{alphanum}
\bibliography{references}

\end{document}